\documentclass[11pt]{article} 

\usepackage[utf8]{inputenc} 
\usepackage{geometry} 
\usepackage{graphicx} 
\usepackage{geometry}
\usepackage{enumerate}
\usepackage{hyperref}
\usepackage{amsmath}
\usepackage{amssymb,amsfonts,amsthm}
\usepackage[all,cmtip]{xy}

\usepackage{booktabs} 
\usepackage{array}
\usepackage{paralist} 
\usepackage{verbatim} 
\usepackage{subfig} 
\usepackage{enumerate}
 
\usepackage{sectsty}
\allsectionsfont{\sffamily\mdseries\upshape} 

\usepackage[nottoc,notlof,notlot]{tocbibind} 
\usepackage[titles,subfigure]{tocloft} 

\newcommand{\R}{\mathbb{R}}
\newcommand{\Q}{\mathbb{Q}}

\newcommand{\Z}{\mathbb{Z}}
\newcommand{\p}{\mathbb{P}}

\newcommand{\T}{\mathbb{T}}

\newcommand{\ext}{\text{Ext}}
\newcommand{\extv}{\ext_{\text{\tiny{TVS}}}}

\newcommand{\id}{\text{Id}}

\newtheorem{theorem}{Theorem}

\newtheorem{lemma}[theorem]{Lemma}
\newtheorem{proposition}[theorem]{Proposition}
\newtheorem{corollary}[theorem]{Corollary}
\theoremstyle{definition}\newtheorem{remark}[theorem]{Remark}
\theoremstyle{definition}
\theoremstyle{definition}
\theoremstyle{definition}
\theoremstyle{definition}\newtheorem{example}[theorem]{Example}
\theoremstyle{definition}

\title{The Ext group in the categories of topological abelian groups and topological vector spaces}
\author{ {Hugo J. Bello }
	\\{\footnotesize Departamento de F\'{\i}sica}\\ {\footnotesize y Matem\'atica Aplicada,}
	\\ \emph{\footnotesize University of Navarra, Spain}
	\\ {\footnotesize {\tt hbello.1@alumni.unav.es}}}

\begin{document}
\maketitle
\begin{abstract}
	This paper is devoted to the study of the group $\ext (G,H)$ of all extensions of topological abelian groups $0\to H\to X\to G\to 0$ and the group $\extv (Z,Y)$ of all extensions of topological vector spaces $0\to Y\to X\to Z\to 0$. We focus on their behaviour under taking products, countable coproducts, dense subgroups   and open subgroups. Finally, we apply the obtained properties  to formulate in a more general setting some known results in the category of locally compact abelian groups   and to determine conditions in which \emph{being a topological vector space} is a three space property.\\
	
	\noindent	\textit{Math. Subj. Class. (2010):} 54H11, 
	22B05, 
	20K35, and  
	57N17. 
\end{abstract}

\section{Introduction}

An extension of an abelian group $G$ by an abelian group $H$ is an abelian group $X\geq H$ such that $X/H\cong G$. In the terminology of homological algebra, an extension of abelian groups is a short exact sequence $0\to H\to X\to G\to 0$. The set   of all extensions of the previous form turns out to  be an abelian group and is one of the objects of study  of homological algebra.

In \cite{mosko} Moskowitz studied for the first time the homological algebra in the class of locally compact abelian groups, which we will denote by $\mathcal L$. Among other things he found the injectivities and projectivities of this class. 
Following his steps, Fulp and Griffith introduced  in  \cite{fulp} the group
 $\ext  (G, H)$ 
 of all extensions $0\to H\to X\to G\to 0$ in     $\mathcal L$. 
 More recently Sahleh and Alijani found several cases in which the group $\ext (G,H)$ is trivial under various algebraic and topological conditions imposed on $G$ and $H$ (\cite{sahleh}, \cite{saal}).
 
In the framework of functional analysis Kalton, Peck and Roberts provided in \cite{kaltonsampler} the first extensive study of the splitting of extensions in the class of complete metrizable topological vector spaces. Several problems involving extensions of different types of topological vector spaces were also considered by Doma\'nski  (\cite{dom0}, \cite{domanski} and \cite{dom2}). 
Latterly Castillo and Sim\~oes (\cite{cast})   studied the limit properties of the functor $\ext$ in the category of complete locally bounded topological vector spaces.

Our purpose is to investigate the group $\ext (G, H)$ of all extensions   $0\to H\to X\to G\to 0$ in the class of topological abelian groups  and the group $\extv (Y, Z)$ of all extensions   $0\to Y\to X\to Z\to 0$ in the class of topological vector spaces. 
In   sections \ref{sect_completions} and \ref{sect_products}   we will study the properties of $\ext$ and $\extv$ when we take products, countable coproducts, dense subgroups and open subgroups. 
We will apply these   properties to the class $\mathcal L$ to show that several theorems proven in \cite{fulp}, \cite{sahleh} and \cite{saal} can be formulated in a more general context.
 
In the fifth section we will use the techniques of section \ref{sect_completions} to discuss the following problem: Given two topological vector spaces $Y$ and $Z$ and an extension of topological abelian groups $0\to Y\to X\to Z\to 0$, find conditions in which we can define in $X$ a compatible topological vector space structure 
in such a way that $0\to  Y\to X\to Z\to 0$ becomes an extension of  topological vector spaces.
This problem was studied by Cattaneo (see \cite{cattaneo}) and by Cabello (see \cite{cabellojmaa}) in several situations in which the completeness of $Z$ is required. We will show that their results can be proved without making use of the completeness of $Z$.

\section{Preliminaries}

All the topological groups   will be Hausdorff.  Since we will  deal only with abelian groups  we will use additive notation.

Given a topological abelian group $G$ and a point $x\in G$ we will use $\mathcal N_x (G)$ to denote a system of neighbourhoods of $x$ in $G$. 

We will denote by $\omega$ the natural numbers, by $\R$ the set of real numbers, by $\p$ the prime numbers, by $\Q_p$ the $p$-adic numbers and by $\Z_p$ the $p$-adic integers.
We will use Greek letters to denote ordinal numbers which will be used as index sets.

Given a topological abelian group $G$, we will call    its completion $\varrho G$, analogously, if $f:G\to H$ is a continuous homomorphism of topological groups, we will call $\varrho f: \varrho G \to \varrho H$  its completion.
A topological  abelian group  is \emph{precompact} if its   completion is a compact group and it is  locally precompact if its  completion is locally compact (\cite[9.13]{RD}).

 A topological abelian group $G$ is called \emph{almost metrizable} if there exist a compact subset $K\subseteq G$ with a countable system of neighbourhoods (see \cite[Section 4.3]{AT}).
  Metrizable topological groups are almost metrizable.
 A topological abelian  group is \v Cech-complete if and only if it is almost metrizable and  complete (\cite[Theorem 4.3.15]{AT}). 
 Every locally compact abelian group is \v Cech-complete (see \cite[Lemma 13.13]{RD}). 
 
  Given $\{G_\alpha:\alpha<\kappa\}$ a family of topological abelian groups 
the \emph{coproduct topology} on the direct sum $\bigoplus_{\alpha<\kappa} G_\alpha$ is the final group topology with respect to the natural inclusions $ G_\gamma \hookrightarrow \bigoplus_{\alpha<\kappa} G_\alpha $. If $\kappa=\omega$ the box topology and the coproduct topology coincide.

 We will work with  topological vector spaces over the field $\R$. A topological vector space (shortly t.v.s.) is called Fr\'{e}chet if it is metrizable,  complete and locally convex.

A short exact sequence of  topological groups  $E: 0\to H\stackrel{\imath}{\to} X \stackrel{\pi}{\to}G\to 0$ will be called a {\it extension of topological groups} (shortly \emph{extension})  when both $\imath$ and $\pi$ are continuous and open onto their images.
Two  extensions of topological abelian groups $E:0\to H\stackrel{\imath}{\to} X \stackrel{\pi}{\to}G\to 0 $ 
and $E\rq{}:0\to H\stackrel{\imath'}{\to} X' \stackrel{\pi'}{\to}G\to 0 $ are said to be equivalent if there exists a 
continuous homomorphism $T:X\to X'$ making the following diagram commutative
\[
\xymatrix{
	E:&	0\ar[r]   &H\ar[r]^\imath\ar@{=}[d]   &X\ar[r]^\pi\ar[d]^T        &G\ar@{=}[d]\ar[r]  &0\\
	E':&	0\ar[r]   &H\ar[r]^{i'}           &X'\ar[r]^{\pi'}                &G\ar[r]  &0
}
\]

(i.~e.~ $T\circ \imath=\imath',$ $\pi'\circ T=\pi$). It is known that if such a $T$ exists, it must actually be a topological isomorphism (the proof of this fact is the same as \cite[Lemma A]{domanski}).

An extension of topological abelian groups  $E:0\to H\stackrel{\imath}{\to} X \stackrel{\pi}{\to}G\to 0 $  \emph{splits} if and only if it is equivalent to the trivial   extension $E_0: 0\to H\stackrel{\imath_0}{\to} H\times G \stackrel{\pi_0}{\to}G\to 0$ where $\imath_0$, $\pi_0$ are the canonical maps and $H\times G$ is endowed with the product topology. 
Note that the extension of topological groups $E$ splits if and only if $\imath (H)$ splits as a subgroup of $X$.

If $X,Y$ and $Z$ are topological vector spaces, a sequence  $E:0\to Y\stackrel{\imath}{\to} X \stackrel{\pi}{\to}Z\to 0 $ is called \emph{an extension of topological vector spaces} if  it is an extension of topological abelian groups and the maps $\imath$ and $\pi$ are also linear.  An extension of topological vector spaces splits if and only if it splits as an extension of topological abelian groups.

The following known characterization is essential when dealing with  extensions of topological groups and topological vector spaces. 
\begin{proposition}\label{theorem_section} Let $E: 0\to H\stackrel{\imath}{\to} X \stackrel{\pi}{\to}G\to 0$ be an extension of topological groups (resp. t.v.s.). The following are equivalent:
	\begin{enumerate}[(i)]
		\item $E$ splits.
		
		\item There exists a a right inverse for $\pi$,  i.e. a continuous homomorphism (linear mapping) $S:G\to X$ with $\pi\circ S={\rm id}_G$.  
		\item There exists a left inverse for $\imath$,  i.e. a   continuous homomorphism (linear mapping) $P:X\to H$ with $P\circ \imath={\rm id}_H$. 
	\end{enumerate}
\end{proposition}

In the following lemmas we introduce the \emph{push-out} and \emph{pull-back} extensions in the categories of topological abelian groups and topological vector spaces. We will not prove these results because the argument is essentially the same   as the one used in abstract abelian groups,   with the obvious replacements. For more details see \cite[Lemmas 1.2, 1.4 and Theorem 2.1 of chapter III]{maclane}.

\begin{lemma}\label{lemma_pullback} 
	Let $E:  0\to H\stackrel{\imath}{\to} X \stackrel{\pi}{\to}G\to 0$ be an extension of topological abelian groups (resp. t.v.s.), let $Y$ be a topological abelian group (t.v.s.), and let $t:Y\to G$ be a continuous homomorphism (linear mapping). The diagram
	\begin{equation}\label{equa_unique2}
	\xymatrix{
		E:&	0\ar[r]   &H\ar[r]^\imath   &X\ar[r]^\pi        &G\ar[r]  &0\\
		&	&           &                                  &Y\ar[u]^t&
	}
	\end{equation}
	can be completed to a diagram of the form
	\[
	\xymatrix{
		E:&	0\ar[r]   &H\ar[r]^\imath   &X\ar[r]^\pi        &G \ar[r]  &0\\
		Et:&	0\ar[r]   &H \ar@{=}[u]\ar[r]^I           &PB\ar[r]^r\ar[u]^s                &Y\ar[u]^t\ar[r]  &0
	}
	\]
	where $PB,$ $r$ and $s$ form the pull-back of $\pi$ and $t$ in the category of topological abelian groups (t.v.s.). The bottom sequence $ Et$ is an extension of topological abelian groups (t.v.s.) which will be called \emph{the pull-back extension}. Furthermore if $\mathcal E:   0\to H\stackrel{I'}{\to} X' \stackrel{r'}{\to}G\to 0$ is another extension of topological abelian groups (t.v.s.) that completes the diagram (\ref{equa_unique2}) in the same way, then $\mathcal E$ and $ Et$ must be equivalent.
\end{lemma}

 From now on we will use the notation $Et$  to denote  the \emph{pull-back extension} of an extension $E$ with respect to a continuous homomorphism $t$.

\begin{lemma}\label{lemma_pushout} 
	Let $E:  0\to H\stackrel{\imath}{\to} X \stackrel{\pi}{\to}G\to 0$ be an extension of topological abelian groups (resp. t.v.s), let $Y$ be a topological abelian group (resp. t.v.s.), and let $t:H\to Y$ be a continuous homomorphism (linear mapping). The diagram
	\begin{equation}\label{equa_unique}
		\xymatrix{
			E:&	0\ar[r]   &H\ar[r]^\imath\ar[d]^t   &X\ar[r]^\pi        &G\ar[r]  &0\\
			&	&Y           &                &&
		}
	\end{equation}
	can be completed to a diagram of the form
	\[
	\xymatrix{
		E:&	0\ar[r]   &H\ar[r]^\imath\ar[d]^t   &X\ar[r]^\pi\ar[d]^s        &G\ar@{=}[d]\ar[r]  &0\\
		tE:&	0\ar[r]   &Y\ar[r]^r           &PO\ar[r]^p                &G\ar[r]  &0
	}
	\]
	where $PO,$ $r$ and $s$ form the push-out of $\imath$ and $t$ in the category of topological abelian groups (t.v.s.). The bottom sequence $tE$ is an extension of topological abelian groups (t.v.s.) which will be called the push-out extension. Furthermore if $\mathcal E:   0\to H\stackrel{r'}{\to} Y \stackrel{p'}{\to}G\to 0$ is another extension of topological abelian groups (t.v.s.) that completes the diagram (\ref{equa_unique}) in the same way, then $\mathcal E$ and $tE$ must be equivalent. 
\end{lemma}

   From now on we will use the notation $tE $  to denote  the \emph{push-out extension} of an extension $E$ with respect to a continuous homomorphism $t$.

Given two topological groups $G$ and $H$, the set of equivalence classes of extensions of topological groups of the form $0\to H\to X\to G\to 0$ will be denoted by $\text{Ext} (G,H)$. We will write $\text{Ext}(G,H)=0$ when every extension $0\to H\to X\to G\to 0$ splits.

Let $E_1: 0\to H\stackrel{\imath_1}{\to} X_1 \stackrel{\pi_1}{\to}G\to 0 $ and  $E_2: 0\to H\stackrel{\imath_2}{\to} X_2 \stackrel{\pi_2}{\to}G\to 0 $ be two extensions of topological abelian groups. Consider the canonical mappings $\nabla_H: H\times H\to H$ $[(h,h')\mapsto h+ h']$  and $\Delta_G: G\to G\times G$ $[g\mapsto (g,g)]$     and the extension of topological abelian groups $E_1\times E_2: 0\to H\times H\stackrel{\imath_1\times \imath_2}{\to} X_1\times X_2 \stackrel{\pi_1\times \pi_2}{\to}G\times G\to 0 $. We will define the sum $E_1 + E_2$ as the extension $\nabla_H ((E_1\times E_2) \Delta_G)$. This operation is called the Baer sum. The set $\ext(G,H)$ with the operation induced by Baer sum in the equivalence classes  of extensions of topological abelian groups is an abelian group (the proof is essentially the same as \cite[Theorem 2.1 of Chapter III]{maclane}). 

 Given two topological vector spaces $Y,Z$ we will call $\extv (Z,Y)$  to the set of equivalence classes of extensions   of topological vector spaces  of the form $   0\to Y\stackrel{\imath}{\to} X \stackrel{\pi}{\to}Z\to 0 $. Taking the operation induced by the Baer sum defined in the same way,  we can endow $\extv (Z,Y)$ with an structure of abelian group.

 We will say that a complete metrizable topological vector space $Z$ is a $\mathcal K$-space if $\extv(Z,\R)=0$ (see \cite[Chapter 5]{kaltonsampler}).

The following lemma will be very useful. The proof is the same as the one in  \cite[Theorem 2.1 of Chapter III]{maclane}.
\begin{lemma}\label{lemm_homo} Let $G$, $H$, $Y_1$ and $Y_2$ be topological abelian groups. Suppose that $t_1: Y_1 \to G$ and $t_2:H\to Y_2$ are continuous homomorphisms. Then the following maps are homomorphisms of abelian groups
  	\[
	\begin{array}{ccccccc}
	\ext (G,H)&\longrightarrow & \ext(Y_1,H)                             &&\ext (G,H)&\longrightarrow & \ext(G,Y_2)\\
	\left[E \right]        &\longmapsto     & \left[  Et_1\right]     &&\left[E \right]        &\longmapsto     & \left[ t_2E\right]
	\end{array}
	\]
	 The analogous statement for topological vector spaces and continuous linear mappings is also true.

\end{lemma}

\section{Properties of the Ext group with respect to  open subgroups and dense subgroups}\label{sect_completions}

\begin{lemma}(\cite[Proposition 3.10]{bcdt})\label{prop_completion_ext}
	Let $E: 0\to H\stackrel{\imath}{\to} X\stackrel{\pi}{\to} G \to 0$ be an  extension of topological abelian groups.  Suppose that the  completion of $H$ is a \v{C}ech-complete group. Then the sequence $\varrho E: 0\to \varrho H\stackrel{\varrho\imath}{\to} \varrho X\stackrel{\varrho\pi}{\to} \varrho G \to 0$ is an extension of topological abelian groups.
\end{lemma}

\begin{theorem}\label{prop_ext_completions_iso}
	Let $G, H$ be topological Abelian groups. If   $H$ is  \v Cech-complete then $\ext (G,H)\cong \ext (\varrho G,  H)$.
\end{theorem}

\begin{proof}
	Consider  the canonical inclusion $\mathcal I: G\to \varrho G$. According to Lemma \ref{lemm_homo} the map
	\[
	\begin{array}{cccc}
	\phi:&\ext (\varrho G,H)&\longrightarrow & \ext(  G,H)\\
	&\left[E \right]        &\longmapsto     & \left[  E\mathcal I\right]
	\end{array}
	\]
	is a homomorphism of abelian groups. Let us see that $\phi$ is an isomorphism.	
	
	To prove that $\phi$ is injective, pick  an extension of topological abelian groups $E :0\to H\stackrel{\imath }{\to} X\stackrel{\pi }{\to}\varrho G\to 0 $ and  suppose that $ E\mathcal I $ splits. 
	An easy verification shows that the sequence $E':0\to   H\stackrel{  {\imath }}{\to}   {\pi^{-1}(G)} \stackrel{ \pi_{|\pi^{-1}(G)} }{\to}  G\to 0 $  is an extension of topological abelian groups. Furthermore the following diagram is commutative
	\[
	\xymatrix{
		E:&	0\ar[r]       &H\ar[r]^\imath   &X\ar[r]^\pi        &\varrho G \ar[r]  &0\\
		E':&	0\ar[r]   &H \ar@{=}[u]\ar[r]^\imath           &\pi^{-1}(G)\ar[r]_{ \;\;\;\;\;\pi_{|\pi^{-1}(G)}}\ar@{^{(}->}[u]                 &G\ar@{^{(}->}[u]^{\mathcal I}\ar[r]  &0
	}
	\]
	According to Lemma \ref{lemma_pullback}, $E'$ must be equivalent to $E\mathcal I$. Then $E'$ splits and applying Proposition \ref{theorem_section} we find a continuous homomorphism $P: \pi^{-1}(G)\to   H$ such that $P\circ   \imath = \id_{{H}}$.
	Since $G$ is dense in $\varrho G$, it is clear  that $\pi^{-1}(G)$ is dense in $X$. Call $R:X\to H$  the canonical extension of $P$ to $X$. $R$ is a continuous homomorphism satisfying $R\circ \imath=\id_H$, hence by Proposition \ref{theorem_section}, $E$ splits.
	
	To check that $\phi$ is onto, choose  an extension of topological abelian groups $\mathcal E   :0\to H\stackrel{I }{\to} Y\stackrel{p }{\to}  G\to 0 $. By Lemma \ref{prop_completion_ext} the sequence $\varrho \mathcal E: 0\to  H\stackrel{\varrho I}{\to} \varrho Y\stackrel{ \varrho p}{\to} \varrho G \to 0$ is an extension of topological abelian groups. The following diagram is commutative
	\[
	\xymatrix{
		\varrho\mathcal	E:&	0\ar[r]       &H\ar[r]^{\varrho I}   &\varrho Y\ar[r]^{\varrho p}        &\varrho G \ar[r]  &0\\
		\mathcal E:&	0\ar[r]   &H \ar@{=}[u]\ar[r]^{  I}           &  Y\ar[r]^{  p}\ar@{^{(}->}[u]                 &G\ar@{^{(}->}[u]^{\mathcal I}\ar[r]  &0
	}
	\]
	
	In virtue of   Lemma \ref{lemma_pullback}, $E$ must be equivalent to $(\varrho \mathcal E)\mathcal I$  hence $\phi(\left[\varrho\mathcal E \right])= \left[\mathcal E\right]$
\end{proof}

\begin{theorem}\label{theo_dense}
	Let $G$ be a topological abelian group and let $H$ be a  \v Cech-complete  topological abelian group. Suppose that $D$ is a dense subgroup of $G$. Then $\ext (G,H)\cong \ext(D,H)$.
\end{theorem}
\begin{proof}
	If $D$ is dense in then $G$ $\varrho D = \varrho G$. By Theorem \ref{prop_ext_completions_iso}
	\[
	\ext (D,H)\cong \ext (\varrho D,   H)= \ext (\varrho G,   H)\cong\ext (G, H).
	\]
\end{proof}

In theorems 3.5 and 3.6 of \cite{fulp} the authors study situations in which $\ext (G,X)=0$ for a fixed $G\in \mathcal L$ and $X$ varying in a subclass of $\mathcal L$. The following corollary is an application of Theorem \ref{prop_ext_completions_iso} on these results.

\begin{corollary}\label{theo_base}
	Let  $G$  be  a locally precompact abelian group.
	\begin{enumerate}[(i)]
		\item $\ext (G, X) = 0$ for all   totally disconnected $X\in \mathcal L$  if and only if $\varrho{G}=(\bigoplus_{\alpha<\kappa} \Z)\times \R^n$ for some $n\in \omega$ and  an arbitrary ordinal number $\kappa$.
		\item $\ext (G,X) = 0$ for all connected $X\in \mathcal L$ if and only if $\varrho{G}=\R^n\times G'$ where  $n\in \Z^{+}$ and $G'$ contains a compact open subgroup having a co-torsion dual. 
 
	\end{enumerate}
\end{corollary}
\begin{proof}
	(i) Suppose that a locally precompact abelian group $G$ has the property that $\ext (G, X) = 0$  for all   totally disconnected $X\in \mathcal L$. Since ever group in $\mathcal L$ is \v Cech-complete, by Theorem \ref{prop_ext_completions_iso} $\ext (\varrho G, X) = \ext (G, X)=0$ for all totally disconnected $X\in \mathcal L$. By  \cite[Theorem 3.5]{fulp},  $\varrho{G}=(\bigoplus_\alpha \Z)\times \R^n$ for some $n\in \omega$ and $\alpha$ an ordinal number. Conversely if $\varrho{G}=(\bigoplus_\alpha \Z)\times \R^n$ by  \cite[Theorem 3.5]{fulp},  $\ext (\varrho G, X) = 0$ for all totally disconnected $X\in \mathcal L$. Using Theorem \ref{prop_ext_completions_iso} we see that $\ext(G,X)=\ext (\varrho G, X) = 0$ for all totally disconnected $X\in \mathcal L$. 
	
	(ii) Apply the same argument to \cite[Theorem  3.6]{fulp}.
  
\end{proof}

In the following examples, we compute several Ext groups:

\begin{example}\label{exam_adele}
	Consider $\Q_d$ the rational numbers with the discrete topology and $D$ any dense subgroup of $\T$. Denote by 
	$\Sigma_a$ the $a$-adic solenoid that satisfies $\Q_d^\wedge \cong \Sigma_a$ (see \cite[25.4]{hr}).
	According to \cite[Exercise 51.7]{fuch1}, $\ext(\Q_d, \Z)\cong \Q^\omega$. By \cite[Theorem 2.12]{fulp},
	$\ext(\Q_d, \Z)\cong \ext(\Z^\wedge, \Q_d^\wedge)\cong \ext (\T,   \Sigma_a)$. Finally, in virtue of  Theorem \ref{theo_dense},
	$\ext (D, \Sigma_a)\cong \Q^\omega$.
\end{example}

\begin{example}
	Let $G$ be a product of locally precompact abelian torsion groups. Consider $0\to \Z\to \R\to \T\to 0$ the canonical extension of $\T$ by $\Z$. It is easy to see that \cite[Theorem 2.14]{fulp} is also valid for topological abelian groups outside the class $\mathcal L$, hence we can consider the classical Hom-Ext exact sequence of abelian groups in this context
	\begin{align*}
		0\to \text{CHom} (G,\Z)\to \text{CHom} (G,\R)&\to \text{CHom}(G,\T)\\
		&\to \ext (G, \Z) \to \ext (G, \R) 
	\end{align*}
	   $\text{CHom} (G,\R)=0$ because $G$ is a torsion group. According to \cite[Corollary 3.15]{bcdt}, since $G$ is a product of locally precompact abelian groups, $\ext(G, \R)=0$. Then $\ext(G,\Z)\cong \text{CHom}(G,\T)$. If we only ask $G$ to be a  topological abelian torsion group, the same argument shows that $\ext(G,\Z)$ contains $\text{CHom}(G,\T)$ as a subgroup.
\end{example}

\begin{theorem}\label{prop_completions_tvs}
	Let $Y, Z$ be topological vector spaces. If $Y$ is complete and metrizable then $\extv (Z,Y)\cong \extv (\varrho Z,  Y)$.
\end{theorem}
\begin{proof}
	Since $Y$ is complete and metrizable, it is \v Cech-complete and we are in the conditions of Lemma \ref{prop_completion_ext}.
	Notice that if we use Lemma \ref{prop_completion_ext} to complete an extension of topological vector spaces we obtain an extension of topological vector spaces. Having this in mind we can repeat the proof of Theorem \ref{prop_ext_completions_iso} in this context.
\end{proof}

The following result generalizes  \cite[Proposition 4.2]{domanski}:
\begin{theorem}\label{theo_dense_tvs}
	Let $Z$ be a topological vector space and let $Y$ be a complete metrizable   topological vector space. If $D$ is a dense subspace of $Z$ then $\extv (Z,Y)\cong \extv(D,Y)$
\end{theorem}
\begin{proof}
	Proceed as in Theorem \ref{theo_dense}   using Theorem \ref{prop_completions_tvs} instead of Theorem \ref{prop_ext_completions_iso}.	
\end{proof}

\begin{lemma}\label{lemm_over} Let $A$ be an open subgroup of a topological group $G$ and suppose that an extension of topological abelian groups
	$E:  0\to H\stackrel{\imath}{\to} X \stackrel{\pi}{\to}A\to 0$
	splits algebraically. Then there exists a group topology $\tau$ on $H\times G$ and an embedding $f:X\to (H\times G,\tau)$ making commutative the diagram
	\begin{equation*}
		\xymatrix{
			\overline E:&	0\ar[r]   &H\ar[r]^{\imath_\tau}\ar@{=}[d]   &(H\times G,\tau)\ar[r]^{\pi_\tau}         &G\ar[r]  &0\\
			E &	0\ar[r]   &H\ar[r]^{\imath}           &X\ar[u]^f\ar[r]^{\pi}                &A\ar@{^{(}->}[u]\ar[r]  &0
		}
	\end{equation*}
	where $\imath_\tau$ and $\pi_\tau$ are the canonical mappings and $\overline E$ is an extension of topological abelian groups.
\end{lemma}
\begin{proof}
	Since $E$ splits algebraically there exists a group topology $\tau'$ on $H\times A$ such that $E$ is equivalent to the extension of topological abelian groups  $\mathcal E: 0\to H\stackrel{\imath_{\tau'}}{\to} (H\times A,\tau') \stackrel{\pi_{\tau'}}{\to}A\to 0$ where $\imath_{\tau'}$ and $\pi_{\tau'}$ are respectively the canonical inclusion and the canonical projection. Call $T$   the topological isomorphism  making the following diagram commutative
	
	\begin{equation}\label{equa_nose}
		\xymatrix{
			\mathcal E:&	0\ar[r]   &H \ar@{=}[d]\ar[r]^{\imath_{\tau'}\;\;\;\;\;}           &(H\times A,\tau')\ar[r]^{\;\;\;\;\;\pi_{\tau'}}                 &A\ar@{=}[d]\ar[r]  &0\\
			E:&	0\ar[r]       &H\ar[r]^\imath   &X\ar[r]^\pi \ar[u]^T        &  A \ar[r]  &0
		}
	\end{equation}
	
	Now, consider on $H\times G$ the group topology $\tau$ obtained by declaring $(H\times A,\tau')$  an open subgroup. An easy verification shows that if we call $\imath_\tau: H\to (H\times G,\tau)$  the canonical inclusion and $\pi_\tau: (H\times G,\tau)\to G$  the canonical projection, the sequence $\overline E: 0\to H\stackrel{\imath_{\tau}}{\to} (H\times G,\tau) \stackrel{\pi_{\tau}}{\to}G\to 0$ is an extension of topological abelian groups. Combining the commutative diagram
	
	\begin{equation*}
		\xymatrix{
			\overline E:&	0\ar[r]   &H\ar[r]^{\imath_\tau\;\;\;\;\;}\ar@{=}[d]   &(H\times G,\tau)\ar[r]^{\pi_\tau}         &G\ar[r]  &0\\
			\mathcal E &	0\ar[r]   &H\ar[r]^{\imath_{\tau'}\;\;\;\;\;}           &(H\times A,\tau')\ar[r]^{\;\;\;\;\;\pi_{\tau'}}\ar@{^{(}->}[u]                 &A\ar@{^{(}->}[u]\ar[r]  &0
		}
	\end{equation*}
	with (\ref{equa_nose})  and defining $f$ as the composition of $T$ and the inclusion $(H\times A,\tau')\hookrightarrow (H\times G,\tau)$, we complete the proof.

\end{proof}

\begin{theorem}\label{teor_open}
	Let $G, H$ be   topological abelian groups. Suppose that $H$ is divisible and that $A$ is an open subgroup of $G$. Then $\ext (G,H)\cong \ext (A, H)$.
\end{theorem}

\begin{proof}
	
	We will use the same strategy as in Theorem \ref{prop_ext_completions_iso}.
	Consider the canonical inclusion $\mathcal I: A\to   G$ . According to Lemma \ref{lemm_homo} the map
	\[
	\begin{array}{cccc}
	\phi:&\ext (  G,H)&\longrightarrow & \ext(  A,H)\\
	&\left[E \right]        &\longmapsto     & \left[  E\mathcal I\right]
	\end{array}
	\]
	is a homomorphism of abelian groups. Let us see that $\phi$ is an isomorphism.	
	
	To prove that $\phi$ is injective, pick  an extension of topological abelian groups $E :0\to H\stackrel{\imath }{\to} X\stackrel{\pi }{\to}  G\to 0 $ and  suppose that $ E\mathcal I $ splits. 
	Since $\imath(H)\leq \pi^{-1}(G)$, 
	the sequence $E':0\to   H\stackrel{  {\imath }}{\to}   {\pi^{-1}(A)} \stackrel{ \pi_{|\pi^{-1}(A)} }{\to}  G\to 0 $ is exact. The mapping $\pi_{|\pi^{-1}(A)}$ is open hence $E'$ is a topological  extension. Furthermore the following diagram is commutative:
	\[
	\xymatrix{
		E:&	0\ar[r]       &H\ar[r]^\imath   &X\ar[r]^\pi        &  G \ar[r]  &0\\
		E':&	0\ar[r]   &H \ar@{=}[u]\ar[r]^\imath           &\pi^{-1}(A)\ar[r]_{ \;\;\;\;\;\pi_{|\pi^{-1}(A)}}\ar@{^{(}->}[u]                 &A\ar@{^{(}->}[u]^{\mathcal I}\ar[r]  &0
	}
	\]
	According to Lemma \ref{lemma_pullback}, $E'$ must be equivalent to $E\mathcal I$. Then $E'$ splits and applying Proposition \ref{theorem_section} we find a continuous homomorphism $P: \pi^{-1}(A)\to   H$ such that $P\circ   \imath = \id_{{H}}$.
	Since $H$ is divisible we can extend the homomorphism $P$ to a homomorphism $R:X\to H$. Since $\pi^{-1}(A)$ is open in $X$ and $R_{|\pi^{-1}(A)} = P$, $R$ is a continuous homomorphism. As $\imath (H)\leq \pi^{-1}(A)$, $R$ satisfies that $R\circ \imath = P\circ \imath = \id_H$   and by Proposition \ref{theorem_section}, $E$ splits.
	
	To check that $\phi$ is onto, choose  an extension of topological abelian groups $E :0\to H\stackrel{I }{\to} Y\stackrel{p }{\to}  A\to 0 $. 
	From   Lemma \ref{lemm_over} we know that there exists a group topology $\tau$ on $H\times G$ and a  commutative  diagram
	\begin{equation*}\label{equa_opened}
		\xymatrix{
			\overline E:&	0\ar[r]   &H\ar[r]^{\imath_\tau}\ar@{=}[d]   &(H\times G,\tau)\ar[r]^{\pi_\tau}         &G\ar[r]  &0\\
			E &	0\ar[r]   &H\ar[r]^{I}           &Y\ar[u]\ar[r]^{p}                &A\ar@{^{(}->}[u]^{\mathcal I}\ar[r]  &0
		}
	\end{equation*}
	where $\imath_\tau$ and $\pi_\tau$ are the canonical mappings and $\overline E$ is an extension of topological abelian groups. In virtue of   Lemma \ref{lemma_pullback}, $ E$ must be equivalent to $(  \overline{  E})\mathcal I$, concluding that $\phi(\left[\overline{  E} \right])=\left[ (\overline{  E}) \mathcal I\right]=   \left[ E\right]$.

\end{proof}

\begin{remark}
	In   \cite[Proposition 3.11]{bcdt} the authors prove that for every \v Cech-complete topological abelian group $H$, if $\ext (\varrho G, H)=0$ then $\ext(G, H)=0$.   Theorem \ref{prop_ext_completions_iso} is a generalization of this fact.
	In \cite[Corollary 14]{bcd} it is proven that if $A$ is an open subgroup of a topological abelian group $G$ and $\ext (A,\T)=0$ then $\ext(G,\T)=0$. Theorem \ref{teor_open} generalizes this result.
\end{remark}

\section{Properties of the $\ext$ group with respect to products and coproducts}\label{sect_products}

\begin{theorem}\label{theo_prod} 
	
	\begin{enumerate}[(i)]
	\item Let $G$ be a topological abelian group and let $\{H_\alpha:\alpha<\kappa\}$ be a family of topological abelian groups. Then   $\ext (G,\prod_{\alpha<\kappa} H_\alpha)\cong \prod_{\alpha<\kappa} \ext (G,H_\alpha)$. 
	 
	 \item 	Let $Z$ be a topological vector space and let $\{Y_\alpha:\alpha<\kappa\}$ be a family of topological vector spaces. Then   $\extv (Z,\prod_{\alpha<\kappa} Y_\alpha)\cong \prod_{\alpha<\kappa} \extv (Z,Y_\alpha)$.
	 \end{enumerate}

\end{theorem}
\begin{proof}
	(i).
  Consider for every $\beta<\kappa$ the canonical projection
	$
	p_\beta:   \prod_{\alpha<\kappa} H_\alpha \to    H_\beta
	$.
	Given    an extension of topological abelian groups $E: 0\to \prod_{\alpha<\kappa} H_\alpha \stackrel{\imath}{\to} X\stackrel{\pi}{\to} G\to 0$, take the push-out extension   $p_\beta E: 0\to  H_\beta \stackrel{r_\beta}{\to} PO_\beta\stackrel{P_\beta}{\to} G\to 0$ with respect to $p_\beta$ and consider the commutative diagram (\ref{equa_prod}) as in Lemma \ref{lemma_pushout}.
	
	\begin{equation}\label{equa_prod}
	\xymatrix{
		E:&	0\ar[r]   & \prod_{\alpha<\kappa} H_\alpha\ar[r]^\imath\ar[d]^{p_\beta}   &X\ar[r]^\pi\ar[d]^{s_\beta}       &G\ar@{=}[d]\ar[r]  &0\\
	p_\beta E:&	0\ar[r]   &H_\beta\ar[r]^{r_\beta}           &PO_\beta\ar[r]^{P_\beta}                &G\ar[r]  &0
	}
	\end{equation}

    In virtue of  Lemma \ref{lemm_homo}  the map
	
	\[
	\begin{array}{cccc}
	\phi: &\ext (G, \prod_{\alpha<\kappa}H_\alpha)&\longrightarrow &  \prod_{\alpha<\kappa} \ext (G,H_\alpha)\\
	         &\left[E \right]        &\longmapsto     & (\left[ p_\alpha E\right])_{\alpha<\kappa}
	\end{array}
	\]
	is a homomorphism of abelian groups. 
	
	Let us check that $\phi$ is injective.
	Take  an extension of topological groups $E: 0\to \prod_{\alpha<\kappa} H_\alpha \stackrel{\imath}{\to} X\stackrel{\pi}{\to} G\to 0$, and suppose that $p_\beta E: 0\to  H_\beta \stackrel{r_\beta}{\to} PO_\beta\stackrel{p_\beta}{\to} G\to 0$ splits for every $\beta<\kappa$. By Proposition \ref{theorem_section} for every $\alpha<\kappa$ there exists a continuous homomorphism $t_\alpha: PO_\alpha \to H_\alpha$ with $t_\alpha\circ r_\alpha= \id_{H_\alpha}$. Define the continuous homomorphism
	\[
	\begin{array}{cccc}
	T:& X&\longrightarrow&  \prod_{\alpha<\kappa} H_\alpha \\
	  & x &\longmapsto   & (t_\alpha(s_\alpha (x)))_{\alpha<\kappa}
	  \end{array}
	\]
	 By the commutativity of (\ref{equa_prod}), $T\circ \imath = \id_{\prod H_\alpha}$, thus $E$ splits.

	To see that $\phi$ is onto pick a family of extensions
	$   \{ E_\alpha :0\to H_\alpha\stackrel{ \imath_\alpha}{\to} X_\alpha\stackrel{\pi_\alpha}{\to}G\to 0 : \alpha<\kappa\}$. 
	Consider the extension $\mathcal E: 0\to \prod_{\alpha<\kappa}H_\alpha\stackrel{ \mathcal I}{\to} \mathcal B \stackrel{\mathcal P}{\to}G\to 0$, where 
	\[\mathcal B=\Big\{\big((x_\alpha)_{\alpha<\kappa}, g\big)\in \prod_{\alpha<\kappa} X_\alpha\times G : \pi_\alpha (x_\alpha)=g\;\forall \alpha<\kappa\Big\},\]
	$\mathcal  I ((h_\alpha)_{\alpha<\kappa})= ((\imath_\alpha (h_\alpha))_{\alpha<\kappa}, 0)$ and $\mathcal P ((x_\alpha)_{\alpha<\kappa}, g)=g$. It is easy to check that $\mathcal E $ is an extension of topological abelian groups. Define for each $\beta <\kappa$ the continuous homomorphism
	\[
	\begin{array}{cccc}
	\mathcal P_\beta:&\mathcal B&\longrightarrow &X_\beta\\
	                 & \big((x_\alpha)_{\alpha<\kappa}, g\big) &\longmapsto & x_\beta
	\end{array}
	\]	
	The following diagram is commutative for every $\beta<\kappa$.
		\[
	\xymatrix{
		\mathcal E:&	0\ar[r]   & \prod_{\alpha<\kappa} H_\alpha\ar[r]^{\mathcal I }\ar[d]^{p_\beta}   &\mathcal B\ar[r]^{\mathcal P}\ar[d]^{\mathcal P_\beta}       &G\ar@{=}[d]\ar[r]  &0\\
		E_\beta:&	0\ar[r]   &H_\beta\ar[r]^{\imath_\beta}           &X_\beta\ar[r]^{\pi_{\beta}}                &G\ar[r]  &0
	}
	\]
	Consequently,   by Lemma \ref{lemma_pushout},  $E_\beta$ must be equivalent to the push-out sequence $p_\beta \mathcal E$. Hence $\varphi (\left[\mathcal E\right])=(\left[E_\alpha\right])_{\alpha<\kappa}$, which concludes the proof.
	
	(ii). All the steps in the previous part remain valid in the category of topological vector spaces so we can proceed in the same way.
\end{proof}

\begin{remark}
In \cite[Proposition 3.12]{bcdt} it is proved that given a family of topological abelian groups $\{H_\alpha: \alpha<\kappa\}$ and a topological group $G$, $\ext(G,H_\alpha)=0$ for every $\alpha<\kappa$ if and only if $\prod_{\alpha<\kappa} \ext (G,H_\alpha)=0$. This result is a particular case of Theorem \ref{theo_prod}. \end{remark}

\begin{corollary}
	Let $G$ be a locally precompact abelian group. Then  $\ext (G, \prod_{p\in\p}\Q^{\alpha_p}_p)$ is a divisible, torsion-free group for every collection of ordinal numbers $\alpha_p, p\in\p$.
\end{corollary}

\begin{proof}
	According to Theorem \ref{theo_prod}
	\[\textstyle
	\ext (G, \prod_{p\in\p}\Q^{\alpha_p}_p)\cong\prod_{p\in\p} \ext (G,  \Q^{\alpha_p}_p)\cong\prod_{p\in\p} \ext (G,  \Q_p)^{\alpha_p}.
	\]
	Since $\Q_p$ is \v Cech-complete,  by  Theorem \ref{prop_ext_completions_iso}  $\ext (G, \Q_p) \cong \ext (\varrho G, \Q_p)$ for every $p\in \p$ and
	\[\textstyle
	\ext (G, \prod_{p\in\p}\Q^{\alpha_p}_p)\cong \prod_{p\in\p} \ext (\varrho  G,  \Q_p)^{\alpha_p}.
	\]
	$\ext (\varrho  G,  \Q_p)$ is divisible and torsion-free by  \cite[Corollary 1.5]{saal}, then   $\ext (G, \prod_{p\in\p}\Q^{\alpha_p}_p)$ is  also divisible and torsion-free.
\end{proof}

\begin{corollary}
	Let  be $G$   a locally precompact abelian group. $\ext (G,X)=0$ for all $X$ product of divisible $\sigma$-compact  groups in $\mathcal L$ if and only if $\varrho G=\R^n\times G'$ where $n$ is a non-negative integer and $G'$ contains a compact open subgroup $K$ such that $K\cong \prod_{\alpha <\alpha_0} \Z/p_\alpha^{r_\alpha}\Z \times \prod_{\beta<\beta_0} \Z_{p_\beta}^{\gamma_\beta}$ where $\alpha_0,\beta_0\in\omega$, $\{\gamma_\beta:\beta<\beta_0\}$ is a family of arbitrary ordinal numbers and $p_\alpha, p_\beta$ are prime numbers for all $\alpha, \beta$.
\end{corollary}

\begin{proof}
	 Suppose that a locally precompact abelian group $G$ has the property that $\ext (G, X) = 0$  for all   $X$ product of divisible $\sigma$-compact  groups in $\mathcal L$. In particular  $\ext (G, X) = 0$ for  all $X$ divisible $\sigma$-compact in $\mathcal L$.  Since every group in $\mathcal L$ is \v Cech-complete, by Theorem \ref{prop_ext_completions_iso}, $\ext (\varrho G, X) = \ext (G, X)=0$, for all divisible $\sigma$-compact $X\in \mathcal L$. According to  \cite[Theorem 2.7]{sahleh},  $\varrho G $ has the desired structure. 
	 
	 Conversely if $\varrho{G}$ has the properties described in the statement, in virtue of  \cite[Theorem 2.7]{sahleh},  $\ext (\varrho G, X) = 0$ for all $X$ divisible $\sigma$-compact  group  $\mathcal L$. 
	 By Theorem  \ref{prop_ext_completions_iso},  $\ext (G, X) = 0$ for every $X$ divisible $\sigma$-compact in $\mathcal L$. Finally, from Theorem \ref{theo_prod} we conclude that the same is true for every $X$ product of divisible $\sigma$-compact  groups in $\mathcal L$.  
	 
\end{proof}

 \begin{corollary}
 	Let $G$ be a torsion group in $\mathcal L$ and let $H$ be product of  divisible torsion-free groups in $\mathcal L$. Then $\ext (G,H)=0$.
 \end{corollary}
 \begin{proof}
 	Suppose that $H=\prod_{\alpha<\kappa} H_\alpha$ with $H_\alpha$ divisible, torsion-free and   $H_\alpha\in \mathcal L$ for every $\alpha<\kappa$.
 	Since $G$ is  locally compact abelian and torsion   we know by \cite[24.21]{hr} that $G$ contains an open compact subgroup $K$. It is clear that $K$ will be a torsion group too.
 	Applying Theorem \ref{theo_prod} and Theorem \ref{teor_open}   
 	\begin{align*}
 		\ext (G,H)\textstyle
 		&\textstyle\cong \ext (G,\prod_{\alpha<\kappa} H_\alpha)\cong \prod_{\alpha<\kappa} \ext( G, H_\alpha)
 		\cong \prod_{\alpha<\kappa}\ext(K,H_\alpha).\\
 	\end{align*}
 	Finally by \cite[Theorem 1.6]{saal} $\ext (K,H_\alpha) =0$ for every $\alpha<\kappa$.

 \end{proof}

\begin{lemma}\label{lemm_coproduct_extensions}
	Let $\{E_\alpha: 0\to H_\alpha \stackrel{\imath_\alpha}{\to} X_\alpha \stackrel{\pi_\alpha}{\to} G_\alpha \to 0 :\alpha<\omega\}$ be a countable family of extensions of topological abelian groups. Consider the coproducts $\bigoplus_{\alpha<\kappa} H_\alpha$, $\bigoplus_{\alpha<\omega} X_\alpha$, $\bigoplus_{\alpha<\omega} G_\alpha$ and the natural mappings $\bigoplus_{\alpha<\omega} \imath_\alpha : \bigoplus_{\alpha<\omega} H_\alpha\to \bigoplus_{\alpha<\omega} X_\alpha$ and $\bigoplus_{\alpha<\omega}   \pi_\alpha: \bigoplus_{\alpha<\omega} X_\alpha \to \bigoplus_{\alpha<\omega} G_\alpha$.
	The sequence \[\textstyle\bigoplus_{\alpha<\omega} E_\alpha: 0\longrightarrow \bigoplus_{\alpha<\omega} H_\alpha \stackrel{\bigoplus_{\alpha<\omega}\imath_\alpha}{\longrightarrow} \bigoplus_{\alpha<\omega}X_\alpha \stackrel{\bigoplus_{\alpha<\omega}\pi_\alpha}{\longrightarrow} \bigoplus_{\alpha<\omega}G_\alpha \longrightarrow 0\] is an extension of topological abelian groups.
\end{lemma}
\begin{proof}
	Straightforward.
\end{proof}

\begin{theorem}\label{theo_coprod}
	Let $H$ be a  topological abelian group and  let $\bigoplus_{\alpha<\omega} G_\alpha$ a countable coproduct of topological abelian groups. Then
	$\ext (\bigoplus_{\alpha<\omega} G_\alpha, H )\cong \prod_{\alpha<\omega} \ext (G_\alpha,H)$. 
\end{theorem}

\begin{proof}
	Consider the canonical inclusion $\mathcal I_\alpha: G_\alpha \to \bigoplus_{\alpha<\omega} G_\alpha$ and define
	\[
	\begin{array}{cccc}
	\phi: & \ext (\bigoplus_{\alpha<\omega} G_\alpha, H)&\longrightarrow & \prod_{\alpha<\omega} \ext (G_\alpha,H)\\
	&\left[E \right]                              & \longmapsto    &\big(\left[E\mathcal I_\alpha \right]\big)_{\alpha<\omega} 
	\end{array}	
	\]
	According to Lemma \ref{lemma_pullback}, $\phi$ is a homomorphism of abelian groups. Let us see that $\phi$ is an isomorphism.
	
	To see that $\phi$ is injective take 	$    E  :0\to H \stackrel{\imath}{\to} X \stackrel{\pi }{\to}\bigoplus_{\alpha<\omega} G_\alpha\to 0$ an extension of topological abelian groups and suppose that $E\mathcal I_\beta$ splits for every $\beta<\omega$. Pick $\beta<\omega$. Take the sequence $    E_\beta :0\to H \stackrel{\imath}{\to} \pi^{-1}(G_\beta) \stackrel{\pi_{|\pi^{-1}(G_\beta)} }{\to}  G_\beta\to 0$. Since $\imath(H)\subset \pi^{-1}(G_\beta)$, $E_\beta$ is an exact sequence. Since  $\pi_{|\pi^{-1}(G_\beta)}$ is open (see \cite[Proposition 2, Chapter 5.1]{bourbaki}) it follows that $E_\beta$ is an extension of topological abelian groups. For every $\beta<\omega$ the following diagram is commutative

	\[
	\xymatrix{
		E:&	0\ar[r]   &H\ar[r]^\imath   &X\ar[r]^\pi        &\bigoplus_{\alpha<\omega} G_\alpha \ar[r]  &0\\
		E_\beta:&	0\ar[r]   &H \ar@{=}[u]\ar[r]^\imath           &\pi^{-1}(G_\beta) \ar[r]^{\;\;\;\;\;\pi_{|\pi^{-1}(G_\beta)}}\ar@{^{(}->}[u]                  &G_\beta\ar@{^{(}->}[u]^{\mathcal I_\beta}\ar[r]  &0
	}
	\]
	Applying Lemma \ref{lemma_pullback} we deduce that $E_\beta$ is equivalent to the pull-back extension $E\mathcal I_\beta$. Then $E_\beta$ splits. By Proposition \ref{theorem_section} for every $\alpha<\omega$ there exist a continuous homomorphism $R_\alpha: G_\alpha \to \pi^{-1}(G_\alpha)$ with $\pi\circ R_\alpha =\id_{G_\alpha}$. Consider the  homomorphism
	\[
	\begin{array}{cccc}
	R: & \bigoplus_{\alpha<\omega} G_\alpha &\longrightarrow & X\\
	& (g_\alpha)_{\alpha<\omega}         &\longmapsto     &\sum_{\alpha<\omega} R_\alpha (g_\alpha)
	\end{array}
	\]
    Using the universal property of the coproduct topology we deduce that $R$ is continuous.
	Since
	\[
	\pi\Big(R\big((g_\alpha)_{\alpha<\omega}\big)\Big)=\pi \big(\sum_{\alpha<\omega} R_\alpha (g_\alpha)\big)= \sum_{\alpha<\omega} \pi \big(  R_\alpha (g_\alpha)\big)=(g_\alpha)_{\alpha<\omega}
	\]
	we obtain that $\pi\circ R=\id_{\bigoplus_{\alpha<\omega}G_\alpha}$ and $E$ splits.
	
	Let us check that $\phi$ is onto. Pick a family of extensions of topological abelian groups $    \{\mathcal E_\alpha :0\to H \stackrel{\imath_\alpha}{\to} X_\alpha \stackrel{\pi_\alpha}{\to}  G_\alpha\to 0 : \alpha<\omega\}$. From Lemma \ref{lemm_coproduct_extensions} we deduce that the sequence \[\textstyle\bigoplus_{\alpha<\omega} \mathcal E_\alpha: 0\longrightarrow \bigoplus_{\alpha<\omega} H \stackrel{\bigoplus_{\alpha<\omega}\imath_\alpha}{\longrightarrow} \bigoplus_{\alpha<\omega}X_\alpha \stackrel{\bigoplus_{\alpha<\omega}\pi_\alpha}{\longrightarrow} \bigoplus_{\alpha<\omega}G_\alpha \longrightarrow 0\] is an extension of topological abelian groups. Define the continuous homomorphism
	\[
	\begin{array}{cccc}
	P: & \bigoplus_{\alpha<\omega} H &\longrightarrow & H\\
	& (h_\alpha)_{\alpha<\omega}         &\longmapsto     &\sum_{\alpha<\omega} h_\alpha
	\end{array}
	\]
	and take the push-out extension  $P\bigoplus_{\alpha<\omega} \mathcal E_\alpha$ as in Lemma \ref{lemma_pushout}. For every $\beta<\omega$ the following diagram is commutative
	\begin{equation}\label{equa_big_diag}
		\xymatrix{
			\mathcal E_\beta:&	0\ar[r]   &H \ar[r]^{\imath_\beta}\ar@{^{(}->}[d]^{}   & X_\beta\ar[r]^{\pi_\beta}\ar@{^{(}->}[d]^{}       &G_\beta\ar@{^{(}->}[d]\ar[r]  &0\\
			 \bigoplus_{\alpha<\omega}\mathcal E_\alpha:&	0\ar[r]   & \bigoplus_{\alpha<\omega}H\ar[d]^P\ar[r]^{\bigoplus \imath_\alpha}           & \bigoplus_{\alpha<\omega}X_\alpha\ar[d]\ar[r]^{\bigoplus_{ }\pi_{\alpha}}                & \bigoplus_{\alpha<\omega}G_\alpha\ar@{=}[d]\ar[r]  &0\\
			  P\bigoplus_{\alpha<\omega} \mathcal E_\alpha:
			&0\ar[r] & H\ar[r] &PO\ar[r] & \bigoplus_{\alpha<\omega}G_\alpha\ar[r]&0 
		}
	\end{equation}
	From Lemma \ref{lemma_pullback} and the commutativity of (\ref{equa_big_diag}) follows that $\mathcal E_\beta$  is equivalent to $(P\bigoplus_{\alpha<\omega} \mathcal E_\alpha)\mathcal I_\beta$ for every $\beta<\omega$ and therefore $\phi ([P\bigoplus_{\alpha<\omega} \mathcal   E_\alpha])= ([\mathcal E_\beta])_{\beta<\omega}$
\end{proof}

\begin{remark}
	It would be interesting to find out if Theorem \ref{theo_coprod} is true for uncountable coproducts of topological abelian groups. Unfortunately,   to prove Lemma \ref{lemm_coproduct_extensions} it is necessary to use that for every countable family of topological abelian groups $\{G_\alpha: \alpha<\omega\}$, the coproduct topology on the direct sum $\bigoplus_{\alpha<\omega} G_\alpha$ coincides with the box topology.
	
	It is worth mentioning that Fulp and Griffith proved in \cite[Theorem 2.13]{fulp} that if $\{G_\alpha: \alpha<\kappa\}$ is a family of groups in $\mathcal L$  such that $(\bigoplus_{\alpha<\kappa}G_\alpha,\tau_{\text{box}})\in \mathcal L$ then $\ext((\bigoplus_{\alpha<\kappa} G_\alpha,\tau_{\text{box}}), H)\cong \prod_{\alpha<\kappa} \ext(G_\alpha, H)$ where $H\in \mathcal L$ and $\tau_{\text{box}}$ is the box topology.
\end{remark}

\section{Relations between $\ext$ and $\extv$}

\begin{theorem}\label{theo_gen_catt}
	Let $E:0\to   Y\stackrel{  {\imath }}{\to}   {X} \stackrel{  {\pi }}{\to}  Z\to 0 $ be an extension of topological abelian groups. Suppose that $Y$ is a Fr\'echet topological vector space and $Z$ is a metrizable topological vector space. Then $X$ admits a compatible topological vector space structure in such a way that $E $ becomes an extension of (metrizable) topological vector spaces.	
\end{theorem}

\begin{proof}  Metrizability is a three space properties (see \cite[5.38(e)]{hr}) hence $X$ is metrizable. 
	
	Since $Y$ is metrizable and complete, it is   \v Cech-complete and we can apply Lemma \ref{prop_completion_ext} to deduce that the completion $\varrho  E: 0\to   Y\stackrel{  {\varrho  \imath }}{\to}   {\varrho  X} \stackrel{  {\varrho  \pi }}{\to}  \varrho  Z\to 0 $ is an extension of topological abelian groups. According to \cite[Proposition 2]{cattaneo} there exist a compatible topological vector space structure in $\varrho  X$ and we can regard $\varrho  E$ as an extension of topological vector spaces. Consider the canonical inclusion $\mathcal I: Z\to \varrho  Z$. The following diagram is commutative
	\[
	\xymatrix{
		\varrho 	E:&	0\ar[r]       &Y\ar[r]^{\varrho  \imath }   &\varrho  X\ar[r]^{\varrho  \pi}        &\varrho  Z \ar[r]  &0\\
		E:&	0\ar[r]   &Y\ar@{=}[u]\ar[r]^{  \imath}           &  X\ar[r]^{  \pi}\ar@{^{(}->}[u]                 &Z\ar@{^{(}->}[u]^{\mathcal I}\ar[r]  &0
	}
	\]
	Since $\varrho  E$ is an extension of topological vector spaces and $\mathcal I$ is a continuous linear mapping, in virtue of Lemma \ref{lemma_pullback}, $E$ is equivalent to the pull-back extension of topological vector spaces $(\varrho  E)\mathcal I: 0\to Y\to PB\to Z\to 0$.  Let $T:X\to PB$ be the topological isomorphism that witnesses the equivalence of $E$ and $(\varrho  E)\mathcal I$. Take in $X$ the topological vector space structure induced by $T$. This completes the proof.
\end{proof}

\begin{theorem}\label{theo_gen_cabello} 
	Let $E:0\to   Y\stackrel{  {\imath }}{\to}   {X} \stackrel{  {\pi }}{\to}  Z\to 0 $ be an extension of topological abelian groups. Suppose that $Y$ is a complete locally bounded topological vector space  and $Z$   is a locally bounded topological vector space. Then $X$ admits a topological vector space structure in such a way that $E $ becomes an extension of (locally bounded) topological vector spaces.	
\end{theorem}

\begin{proof} $X$ is locally bounded because local boundedness is a three space property (see \cite[Theorem 3.2]{rodi3}). 
$Y$ is metrizable because every locally bounded Hausdorff topological vector space is metrizable. 
Hence $Y$ is in particular almost metrizable  and we can apply Lemma \ref{prop_completion_ext} to deduce that the completion $\varrho E: 0\to   \varrho Y\stackrel{  {\varrho \imath }}{\to}   {\varrho X} \stackrel{  {\varrho \pi }}{\to}  \varrho Z\to 0 $ is an extension of topological abelian groups. Since $Z$ is locally bounded, its completion $\varrho Z$ is also locally bounded and we are in the conditions of  \cite[Theorem 4]{cabellojmaa}. So $\varrho E$ can be regarded as an extension of topological vector spaces. From here proceed as in the proof of Theorem \ref{theo_gen_catt}.
\end{proof}

\begin{remark}
	Notice that in theorems    \ref{theo_gen_catt} and 	\ref{theo_gen_cabello}, when we construct the topological vector space structure on $X$, the inclusion $X \hookrightarrow \varrho X$ turns out to be a linear map. Consequently, the   multiplication by scalars on $X$, say $*_X: \R\times X\to X$, can be regarded as the restriction of the multiplication by scalars on $\varrho X$, say $*_{\varrho X}:  \R\times \varrho X\to \varrho X$.
	Completing the topological vector space $(X,*_X)$ we obtain another topological vector space $(\varrho X, \varrho (*_X))$. Thus the multiplication by scalars $\varrho (*_X): \R\times \varrho X\to \varrho X$ coincides with $*_X$ in $\R\times X$. Finally, since the continuous homomorphism $\varrho (*_X)$ coincides with the continuous homomorphism $*_{\varrho X}$ in the dense subgroup $\R\times X$, we conclude that both operations are the same. 
\end{remark}

\begin{corollary}
	Let $Y$ and $Z$ be topological vector spaces.
	\begin{enumerate}[(i)]
	 \item If $Y$ is Fr\'echet and $Z$ is metrizable then $\ext(Z,Y)\cong\extv(Z,Y)$.
	 \item If $Y$  is complete locally bounded and $Z$ is locally bounded then $\ext(Z,Y)\cong\extv(Z,Y)$
	 \end{enumerate}
\end{corollary}

 \begin{corollary}\label{prop_cattaneo}
 	Let $X$ be an abelian topological group and let $\pi:X\to Z$ be an open continuous homomorphism of $X$ onto   a topological vector space $Z$.
 	\begin{enumerate}[(i)]
   \item	 If $\ker \pi$ is a   Fr\'echet space and $Z$ is metrizable then $X$ is a (metrizable)  topological vector space.
   \item    If $\ker \pi$ is a complete locally bounded topological vector space and $Z$ is locally bounded then $X$ is a (locally bounded) topological vector space.
    
    \end{enumerate}
 \end{corollary}
 
\begin{corollary}\label{prop_kspaces}
	A   metrizable topological vector space   $Z$   satisfies  $\ext(Z, \R)=0$ if and only if $\varrho Z$ is a $\mathcal K$-space.
\end{corollary}
\section*{Acknowledgements} 
\indent

The author gratefully acknowledges the many helpful suggestions of Mar\'\i a Jes\'us Chasco, Xabier Dom\'inguez and Mikhail Tkachenko during the preparation of the paper. The author thanks the referee as well for the careful report. Finally, the author wishes to thank the Asociaci\'on de Amigos de la Universidad de Navarra and the Spanish Ministerio de Econom\'{\i}a y Competitividad (grant MTM 2013-42486-P) for their financial support.

\end{document}